\documentclass{amsart}

\usepackage{amsmath, amsthm, amsfonts, amssymb}
\usepackage{latexsym}
\usepackage{setspace}
\usepackage[pdftex]{graphicx}
\usepackage{enumitem}
\usepackage{subfigure}
\usepackage{hyperref}
\usepackage{color}
\usepackage{pinlabel}

\theoremstyle{plain}
\newtheorem{theorem}{Theorem}

\newtheorem{lemma}[theorem]{Lemma}
\newtheorem{proposition}[theorem]{Proposition}
\newtheorem{conjecture}[theorem]{Conjecture}
\newtheorem{question}[theorem]{Question}

\theoremstyle{definition}
\newtheorem{definition}[theorem]{Definition}
\newtheorem*{acknowledgements}{Acknowledgements}
\newtheorem*{remark}{Remark}

\makeatletter
\g@addto@macro\@floatboxreset\centering
\makeatother

\DeclareMathOperator{\rk}{rk}
\DeclareMathOperator{\HFhat}{\widehat{HF}}
\DeclareMathOperator{\Homeo+}{Homeo^+}

\newcommand{\numset}[1]{\mathbb{#1}}

\newcommand{\Z}{\numset{Z}}
\newcommand{\Q}{\numset{Q}}
\newcommand{\R}{\numset{R}}

\definecolor{a-color}{RGB}{255,0,165}
\definecolor{b-color}{RGB}{0,165,255}
\definecolor{c-color}{RGB}{255,89,0}
\definecolor{d-color}{RGB}{0,255,89}
\definecolor{greengen-color}{RGB}{0,255,0}
\definecolor{redgen-color}{RGB}{255,0,0}
\definecolor{bluegen-color}{RGB}{0,0,255}
\definecolor{mu-color}{RGB}{165,0,255}
\definecolor{s-color}{RGB}{255,165,0}
\definecolor{ab-color}{RGB}{133,204,0}

\title{Non-left-orderable surgeries on twisted torus knots}

\author[K.~Christianson]{Katherine Christianson}
\address{Department of Mathematics, Columbia University, New York, NY 10027, 
  USA
}
\email{\href{mailto:mac2370@columbia.edu}{mac2370@columbia.edu}}

\author[J.~Goluboff]{Justin Goluboff}
\address{Department of Mathematics, Columbia University, New York, NY 10027, 
  USA
}
\email{\href{mailto:jg3221@columbia.edu}{jg3221@columbia.edu}}

\author[L.~Hamman]{Linus Hamann}
\address{Department of Mathematics, Columbia University, New York, NY 10027, 
  USA
}
\email{\href{mailto:dlh2145@columbia.edu}{dlh2145@columbia.edu}}

\author[S.~Varadaraj]{Srikar Varadaraj}
\address{Department of Mathematics, Columbia University, New York, NY 10027, 
  USA
}
\email{\href{mailto:sv2423@columbia.edu}{sv2423@columbia.edu}}

\subjclass[2010]{57M25 (20F60 57M50)}
\keywords{left-orderability, twisted torus knots, Dehn surgery}

\begin{document}
\begin{abstract}
Boyer, Gordon, and Watson \cite{BGW13} have conjectured that an
irreducible rational homology 3-sphere is an L-space if and only if
its fundamental group is not left-orderable. Since large classes of
L-spaces can be produced from Dehn surgery on knots in $S^3$, it is 
natural to ask what conditions on the knot group are sufficient to
imply that the quotient associated to Dehn surgery is not left-orderable.
Clay and Watson develop a criterion for determining the left-orderability
of this quotient group in \cite{CW13} and use it to verify the conjecture
for surgeries on certain L-space twisted torus knots. We generalize a recent
theorem of Ichihara and Temma \cite{IT14} to provide another such criterion. We 
then use this new criterion to generalize the results of Clay and Watson and 
to verify the conjecture for a much broader class of L-space twisted torus knots. 
\end{abstract}
\maketitle

\section{Introduction}
For a closed, connected, orientable 3-manifold $Y$, let $\HFhat(Y)$
denote the Heegaard Floer homology of $Y$, as defined in \cite{OSH04}.
We begin with a definition.
\begin{definition}
A closed, connected, orientable 3-manifold $Y$ is an \emph{L-space}
if it is a rational homology sphere satisfying 
$\rk\HFhat(Y)=|H_1(Y;\Z)|$.
\end{definition}
A result due to Ozsv\'ath and Szab\'o \cite[Proposition 5.1]{OSA04} gives
$\rk\HFhat(Y)\geq|H_1(Y;\Z)|$ for any $Y$. Thus, we can understand
L-spaces as spaces with minimal Heegaard Floer homology. L-spaces
derive their name from lens spaces, which were the first class of
spaces observed to have minimal Heegaard Floer homology; however, many other
spaces, such as those which admit an elliptic geometry
\cite[Proposition 2.3]{OS05} are also L-spaces. 

It is interesting to consider whether L-spaces may be characterized using
properties unrelated to their Heegaard Floer homologies. We recall the
following definition.
\begin{definition}
A nontrivial group $G$ is \emph{left-orderable} if there exists a strict total
ordering $>$ of the elements of $G$ that is left-invariant: whenever $g>h$
then $fg>fh$, for all $g,h,f\in G$. 
\end{definition}
Boyer, Gordon, and Watson established that a closed, connected,
Seifert fibred 3-manifold is an L-space if and only if its
fundamental group cannot be left-ordered \cite{BGW13}. After providing further examples
to support this correspondence, they proposed the following conjecture.
\begin{conjecture}[{\cite[Conjecture 3]{BGW13}}] \label{conj:BGW}
An irreducible rational homology 3-sphere is an L-space if and only
if its fundamental group is not left-orderable. 
\end{conjecture}
In order to investigate this conjecture, it is useful to consider
Dehn surgery on knots in $S^3$, since this process provides large classes
of 3-manifolds. Boyer, Rolfsen, and Weist \cite[Theorem 1.1]{BRW05} 
demonstrated that the fundamental group of a $P^2$-irreducible, connected,
compact $3$-manifold is left-orderable if and only if it has a nontrivial homomorphic
image which is left-orderable. Since the abelianization of any knot group is
$\Z$, we have that any knot group is left-orderable. However, the fundamental
group of a manifold produced by Dehn surgery is a quotient of the knot group,
which may or may not be left-orderable.
In light of these observations and Conjecture~\ref{conj:BGW}, it is
natural to ask the following question (cf. \cite[Question 1.4]{CW13}).
\begin{question}
Given a knot $K$ in $S^3$ and a rational number $r$, what conditions
on the knot group of $K$ are sufficient to imply that $r$-surgery on
$K$ yields a manifold with non-left-orderable fundamental group?
\label{question}
\end{question}
In \cite{CW13}, Clay and Watson answer Question~\ref{question} with
the following sufficient condition. We denote by $S_K^3(r)$
the manifold produced by $r$-surgery on a knot $K$. 
\begin{theorem}[{\cite[Theorem 1.5]{CW13}}]
Let $K$ be a nontrivial knot in $S^3$, let $\mu$ and $\lambda$ be a
meridian and 0-framed longitude, respectively, of $K$, and let 
$\frac{p_0}{q_0}, \frac{p_1}{q_1} \in \Q^+$ with $p_i, q_i > 0$. If 
$\mu^{p_0}\lambda^{q_0} > 1$ implies $\mu^{p_1}\lambda^{q_1} > 1$
for every left ordering $>$ of the knot group of $K$, then 
$\pi_1(S_K^3(p/q))$ is not left-orderable for any $\frac{p}{q} \in \Q^+$
such that $p, q > 0$ and
$\frac{p}{q} \in \left( \frac{p_0}{q_0}, \frac{p_1}{q_1} \right)$. 
\label{thm:CWCond}
\end{theorem}
In order to consider other sufficient conditions that answer
Question~\ref{question}, we require the following well-known equivalent
condition for left-orderability (see, for instance, \cite[Theorem~6.8]{Ghy01}).
\begin{theorem}
Let $G$ be a countable group. Then the following are equivalent:
\begin{itemize}
\item $G$ acts faithfully on the real line by order-preserving homeomorphisms.
\item $G$ is left-orderable.
\label{thm:LOCond}
\end{itemize}
\end{theorem}
Let us denote by $\Homeo+(\R)$ the group of order-preserving homeomorphisms
of $\R$. Then, the first condition in Theorem~\ref{thm:LOCond} is equivalent
to the existence of an injective homomorphism
$\Phi: G \rightarrow \Homeo+(\R)$. For such homomorphisms, we will sometimes abuse
notation and write $gt$ for $\Phi(g)t$ for elements $g \in G$ and $t \in \R$.

We are interested in studying global fixed points of such a homomorphism,
i.e. points $t \in \R$ such that $\Phi(g)t = t$ for all $g \in G$. The following
lemma due to Boyer, Rolfson, and Weist demonstrates the importance of these points.
\begin{lemma}[{\cite[Lemma 5.1]{BRW05}}]
If there is a homomorphism $\Phi: G \rightarrow \Homeo+(\R)$ with nontrivial image,
then there is another such homomorphism with no global fixed points.
\label{lem:BRW}
\end{lemma}
With this lemma and Theorem~\ref{thm:LOCond}, the following criterion for
non-left-orderability is straightforward.
\begin{proposition}
If $G$ is a countable group and every homomorphism $\Phi: G \rightarrow
\Homeo+(\R)$ has a global fixed point, then $G$ is not left-orderable.
\label{prop:NonLOCond}
\end{proposition}
\begin{proof}
By contradiction, assume $G$ is left-orderable. Then, by Theorem~\ref{thm:LOCond},
there exists $\Phi: G \rightarrow \Homeo+(\R)$ injective. We can then apply
Lemma~\ref{lem:BRW} to conclude that there exists 
$\Phi': G \rightarrow \Homeo+(\R)$ with no global fixed points.
But this contradicts our hypothesis.
\end{proof}
Ichihara and Temma use exactly the reasoning of Proposition~\ref{prop:NonLOCond}
in \cite{IT14} to demonstrate the following criterion for non-left-orderability
of the fundamental groups of surgery manifolds. Their work was motivated by 
that of Nakae in \cite{Nak13}.
\begin{theorem}[\cite{IT14}]
Let $K$ be a knot in $S^3$. Suppose that the knot group $\pi_1(S^3 - K)$
is of the form
\[\langle x,y \mid (w_1x^{c}w_1^{-1})y^{-r}(w_2^{-1}x^{d}w_2)y^{r-\ell}, \mu\lambda\mu^{-1}\lambda^{-1} \rangle\]
where $c,d \geq 0$, $r \in \Z$, $\ell \geq 0$, $\mu = x$, 
$\lambda = x^{-m}wx^{-n}$, $w$ is a word which excludes $x^{-1}$ and
$y^{-1}$, $m,n \geq 0$, and $p/q \geq m+n$.
Then, Dehn surgery along the slope $p/q$ yield a closed 3-manifold
with non-left-orderable fundamental group.
\label{thm:ITCond}
\end{theorem}
The criteria used in the proof of Theorem~\ref{thm:ITCond} actually apply
to a more general class of group presentations. The main result of our paper is
an extraction of these criteria which reframes the theorem in a more widely
applicable manner. The proof of Theorem~\ref{thm:CGHVCond} closely follows the
proof of Theorem~\ref{thm:ITCond}.
\begin{theorem} \label{thm:CGHVCond}
Let $K$ be a nontrivial knot in $S^3$. Let G denote the knot group of $K$, and
let $G(p/q)$ be the quotient of G resulting from $p/q$-surgery. Let $\mu$ be a
meridian of $K$ and $s$ be a $v$-framed longitude with $v>0$. Suppose that
$G$ has two generators, $x$ and $y$, such that $x=\mu$ and $s$ is a word which
excludes $x^{-1}$ and $y^{-1}$ and contains at least one $x$. Suppose further
that every homomorphism $\Phi: G(p/q) \rightarrow \Homeo+(\R)$ satisfies 
$\Phi(x)t > t$ for all $t \Rightarrow \Phi(y)t \geq t$ for all $t$.
If $p,q > 0$, then, for $p/q \geq v$, $G(p/q)$ is not left-orderable.
\end{theorem}
\begin{remark}
When applying Theorem~\ref{thm:CGHVCond}, it is sufficient to demonstrate
that every homomorphism $\Phi: G \rightarrow \Homeo+(\R)$ satisfies
$\Phi(x)t > t$ for all $t \Rightarrow \Phi(y)t \geq t$ for all $t$, since this
implies the final hypothesis in the statement of the theorem.
As a result, we can understand Theorem~\ref{thm:CGHVCond}, like Theorems
\ref{thm:CWCond} and \ref{thm:ITCond}, to be a set of conditions on the knot group.
\end{remark}
\begin{proof}[Proof of Theorem~\ref{thm:CGHVCond}]
By Proposition~\ref{prop:NonLOCond}, it suffices to show that every
homomorphism $\Phi: G(p/q) \rightarrow \Homeo+(\R)$ has a global fixed point.
First, note that since $G(p/q)$ has the relation $x^{p-qv}s^q = 1$, we have
\begin{equation}
\label{eq:surgrel}
x^{qv-p} = s^q.
\end{equation}
Now, assume that $xt = t$ for some $t \in \mathbb{R}$. Assume 
$yt \neq t$; then, we can pick an order such that $yt > t$, 
or equivalently, $y^{-1}t < t$. By hypothesis, $s^{-1}$ contains only 
$x^{-1}$ and $y^{-1}$, so we have $x^{-1}t = t$, $y^{-1}t < t$,
and $x^{-1}y^{-1}t < x^{-1}t = t$. Thus, $s^{-q}t < t$. (Note that $s$ must
contain at least one $y$: otherwise, $s$ would be a power of the meridian $x$, so
$s$ and $x$ could not generate the peripheral subgroup.) But then
\[t > s^{-q}t = x^{p-qv}t = t\]
which is a contradiction. Thus, $yt = t$, and we have a global fixed point.

Now, we are left with the case $xt \neq t$ for all $t$. We prove
this is impossible. Since $x$ has no fixed points, we can pick an
order such that $xt > t$ for all $t$. By assumption, then, 
$yt \geq t$ for all $t$. Now, $s$ contains only $x$ and $y$, and 
we have $yxt > yt \geq t$. So, $st > t$ for all $t$, since
$s$ contains at least one $x$ by assumption. By Equation
\eqref{eq:surgrel}, $s^qt = x^{qv-p}t$, so we must have $qv - p > 0$,
or $v > p/q$. But this contradicts the assumption $p/q \geq v$.
\end{proof}
We note that this theorem can be restated in a purely group-theoretic sense.
Consider a group $G$ that has a $\Z \oplus \Z$-subgroup with distinguished generators
$\mu$ and $s$. If we assume the hypotheses of Theorem~\ref{thm:CGHVCond}, then
$G/\langle \langle \mu^{p-qv}s^q \rangle \rangle$ is not left-orderable.

Returning to the motivation for Question~\ref{question}, we can consider
how these various criteria for non-left-orderability of the fundamental
groups of surgery manifolds can help verify Conjecture~\ref{conj:BGW} for
these manifolds. We say that a knot $K$ in $S^3$ which admits a positive
L-space surgery is an \emph{L-space knot}. It is known \cite[Corollary 1.4]{OS11}
that if $K$ is an L-space knot, then $r$-surgery on $K$ produces an L-space exactly
when $r \geq 2g(K)-1$, where $g(K)$ denotes the genus of $K$.
We will focus on demonstrating that surgeries larger than this bound on known 
L-space knots produce manifolds with non-left-orderable fundamental groups.

The specific knots we will consider are families of L-space twisted torus knots.
We denote by $T_{p,q}^{\ell,m}$ the twisted torus knot 
obtained from the $(p,q)$-torus knot by twisting $\ell$ strands
$m$ full times. We will call this twisted torus knot the 
\emph{$(p,q,\ell,m)$-twisted torus knot}. Figure \ref{fig:twisted_ex}, for
instance, shows $T_{5,6}^{2,2}$.

\begin{figure}
\centering
\includegraphics[width=\textwidth]{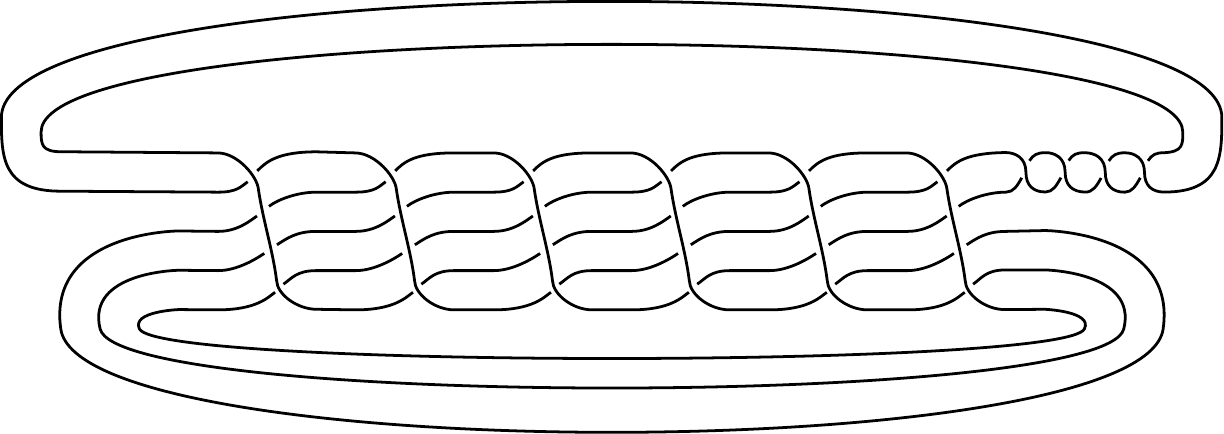}
\caption{The $(5,6,2,2)$-twisted torus knot, denoted $T_{5,6}^{2,2}$.}
\label{fig:twisted_ex}
\end{figure}

We denote the 3-manifold produced by $r$-surgery on 
$T_{p,q}^{\ell,m}$ as $M_{p,q}^{\ell,m}(r)$. Throughout, we will
assume $p, q, \ell, m, r > 0$, $r\in \Q$.

The following theorem due to Vafaee \cite{Vaf13} provides a class of twisted torus
knots which are known to be L-space knots.
\begin{theorem}[\cite{Vaf13}]
\label{thm:Vaf}
$T_{p,pk \pm 1}^{\ell,m}$ is an L-space knot if and only if
\begin{enumerate}
\item $\ell = p-1$
\item $\ell = p-2$ and $m = 1$, or
\item $\ell = 2$ and $m = 1$.
\end{enumerate}
\end{theorem}
In \cite{CW13}, Clay and Watson apply Theorem~\ref{thm:CWCond} to
certain subfamilies of the L-space knots specified in 
Theorem~\ref{thm:Vaf}. Their progress is summarized by the following
theorem.
\begin{samepage}
\begin{theorem}[{\cite[Theorems 4.5 and 4.7]{CW13}}] \label{thm:CW}
\mbox{}
\begin{enumerate}[label=(\alph*)]
\item For sufficiently large $r$, $M_{3,5}^{2,m}(r)$ has a non-left-orderable fundamental group.
\item Suppose that $q$ is a positive integer congruent to $2$ modulo $3$.
For sufficiently large $r$, $M_{3,q}^{2,1}(r)$ has a non-left-orderable 
fundamental group.
\end{enumerate}
\end{theorem}
\end{samepage}
We note that the knots considered by Clay and Watson are all of the
form $T_{p,p(k+1) - 1}^{p-1,m}$. Thus, both theorems
verify Conjecture~\ref{conj:BGW} for surgeries on subfamilies of the
first case of twisted torus knots specified by Theorem~\ref{thm:Vaf}.

In \cite{IT14}, Ichihara and Temma similarly apply Theorem \ref{thm:ITCond} to
prove the following theorem, generalizing the work of Clay and Watson.
\begin{theorem}[{\cite[Corollary 1.2]{IT14}}]
\label{thm:ITCor}
Suppose that $q$ is a positive integer congruent to $2$ modulo $3$. For
sufficiently large $r$, then $M_{3,3k-1}^{2,m}(r)$ has non-left-orderable
fundamental group. \end{theorem} In this paper, we apply
Theorem~\ref{thm:CGHVCond} to prove the following result.
\begin{theorem}
\label{thm:CGHV1}
For sufficiently large $r$, $M_{p,pk \pm 1}^{p-1,m}(r)$ and
$M_{p,pk \pm 1}^{p-2,1}(r)$ have non-left-orderable fundamental group.
\end{theorem}
Theorem~\ref{thm:CGHV1} answers Question \ref{question} for
cases 1 and 2 of the twisted torus knots specified in Theorem~\ref{thm:Vaf}.
Case 1 is a generalization of Theorems~\ref{thm:CW} and \ref{thm:ITCor}; case
2 is an entirely new family of twisted torus knots. In light of
Theorem~\ref{thm:Vaf}, these results support Conjecture~\ref{conj:BGW}.

The paper is organized as follows. In Section \ref{sec:Groups}, we compute
the knot groups of $T_{p,pk \pm 1}^{\ell,m}$ and their corresponding
peripheral subgroups. In Section \ref{sec:LO}, we prove
Theorem~\ref{thm:CGHV1}.

\begin{acknowledgements}
We thank Jennifer Hom and Mike Wong for providing the resources and
background material necessary for our work and for their guidance
throughout the research process. We also thank Liam Watson for his 
support and for useful discussions. We thank the Columbia University
math REU program for giving us the opportunity to pursue this research.
This REU program was partially funded by NSF grant DMS-0739392.
\end{acknowledgements}

\section{Computing Knot Groups and Peripheral Subgroups}
\label{sec:Groups}
First, we fix some notation. For the knot groups of twisted torus
knots, we will generalize the notation of \cite{CW13} by defining 
\[G_{p,q}^{\ell,m} = \pi_1(S^3 - T_{p,q}^{\ell,m}).\]
We will denote by $G_{p,q}^{\ell,m}(r)$ the fundamental group of
$M_{p,q}^{\ell,m}(r)$.

It is a well-known fact (see, for instance, \cite{Lic97}) 
that, for a nontrivial knot, the fundamental group of the 
boundary of the knot complement injects into the knot group.
Its image (up to conjugation) is the peripheral subgroup, 
which means that the peripheral subgroup is abelian.

We now derive the knot groups of two general cases, 
$T_{p,pk+1}^{\ell,m}$ and $T_{p,pk - 1}^{\ell,m}$. From these,
we can then get the knot groups of the L-space twisted torus
knots specified in Theorem \ref{thm:Vaf} by plugging in specific
values of $\ell$ and $m$. Our approach to both cases is the same
as that of Clay and Watson \cite{CW13}: we will use the
Seifert-van Kampen Theorem applied to a genus-two Heegaard splitting,
with the knot appearing on the Heegaard surface
(see Figures \ref{fig:pk-1_surfacegen} and \ref{fig:pk+1_surfacegen}, which depict
$T_{5,5k-1}^{2,m}$ and $T_{5,5k+1}^{2,m}$ as examples). 

\begin{figure}
\labellist
\pinlabel ${\LARGE \underbrace{ }}$ at 125 187
\pinlabel {\small $m$ twists} at 125 173
\pinlabel $\LARGE \overbrace{ }$ at 333 160
\pinlabel {\small $k$ twists} at 333 175
\endlabellist
\centering
\includegraphics[width=\textwidth]{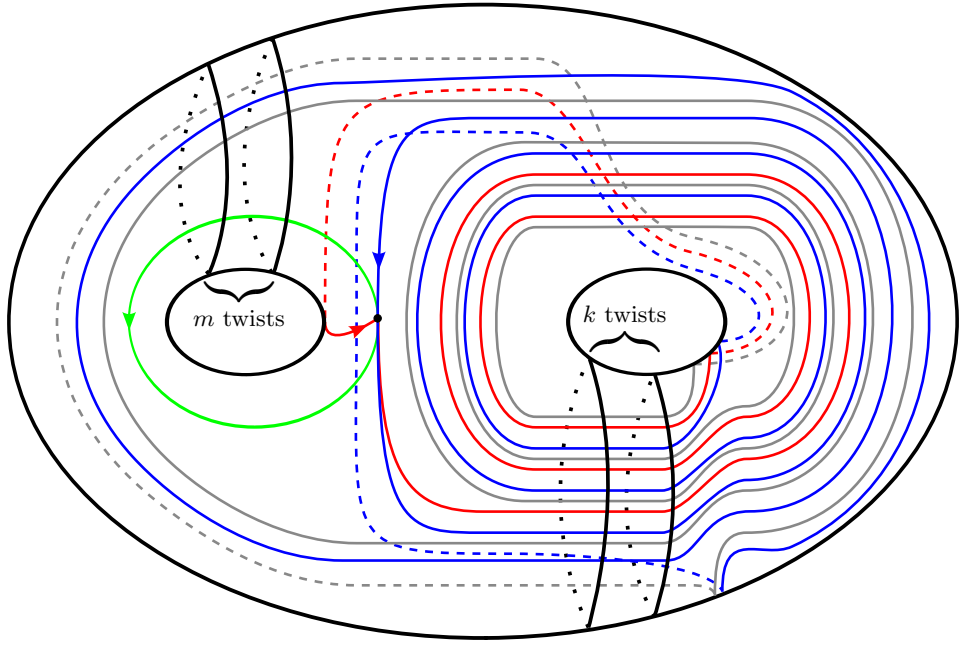}
\caption{Generators for the fundamental group of
$\Sigma \setminus \nu (T_{p,pk-1}^{\ell,m})$. Here, we show
$T_{5,5k-1}^{2,m}$ as an example. The knot is shown in gray.}
\label{fig:pk-1_surfacegen}
\end{figure}
\begin{figure}
\labellist
\pinlabel \textcolor{a-color}{a} at 383 80
\pinlabel \textcolor{b-color}{b} at 290 185
\pinlabel \textcolor{c-color}{c} at 128 260
\pinlabel \textcolor{d-color}{d} at 220 185
\endlabellist
\centering
\includegraphics[width=\textwidth]{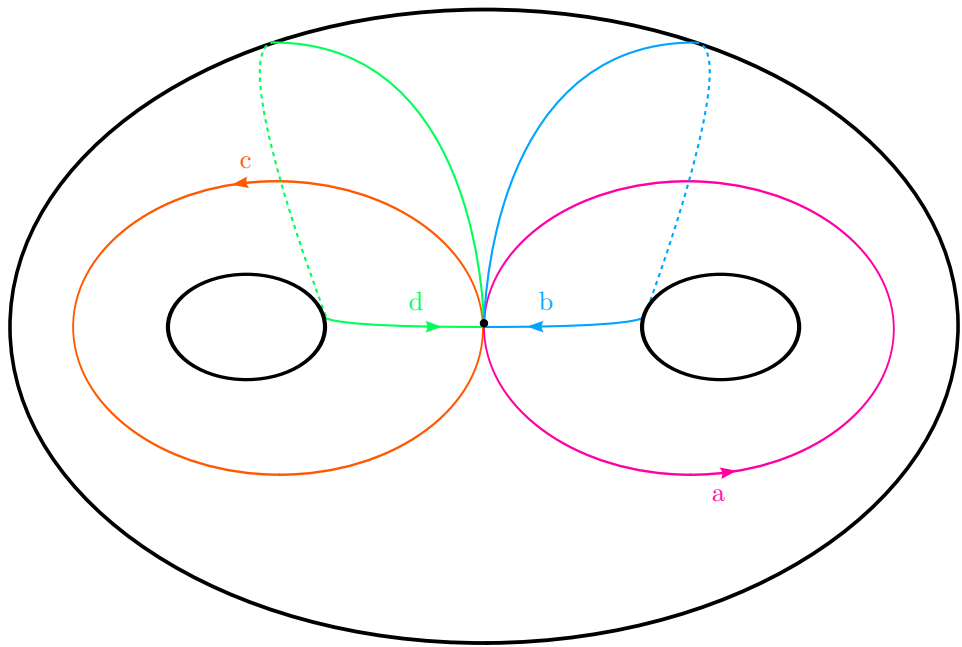}
\caption{Generators for the fundamental group of the 
Heegaard surface $\Sigma$.}
\label{fig:pk-1_generators}
\end{figure}
\begin{figure}
\labellist
\pinlabel ${\LARGE \underbrace{ }}$ at 125 187
\pinlabel {\small $m$ twists} at 125 173
\pinlabel $\LARGE \overbrace{ }$ at 333 160
\pinlabel {\small $k$ twists} at 333 175
\pinlabel \textcolor{mu-color}{$\mu$} at 210 180
\pinlabel \textcolor{s-color}{s} at 185 185
\endlabellist
\centering
\includegraphics[width=\textwidth]{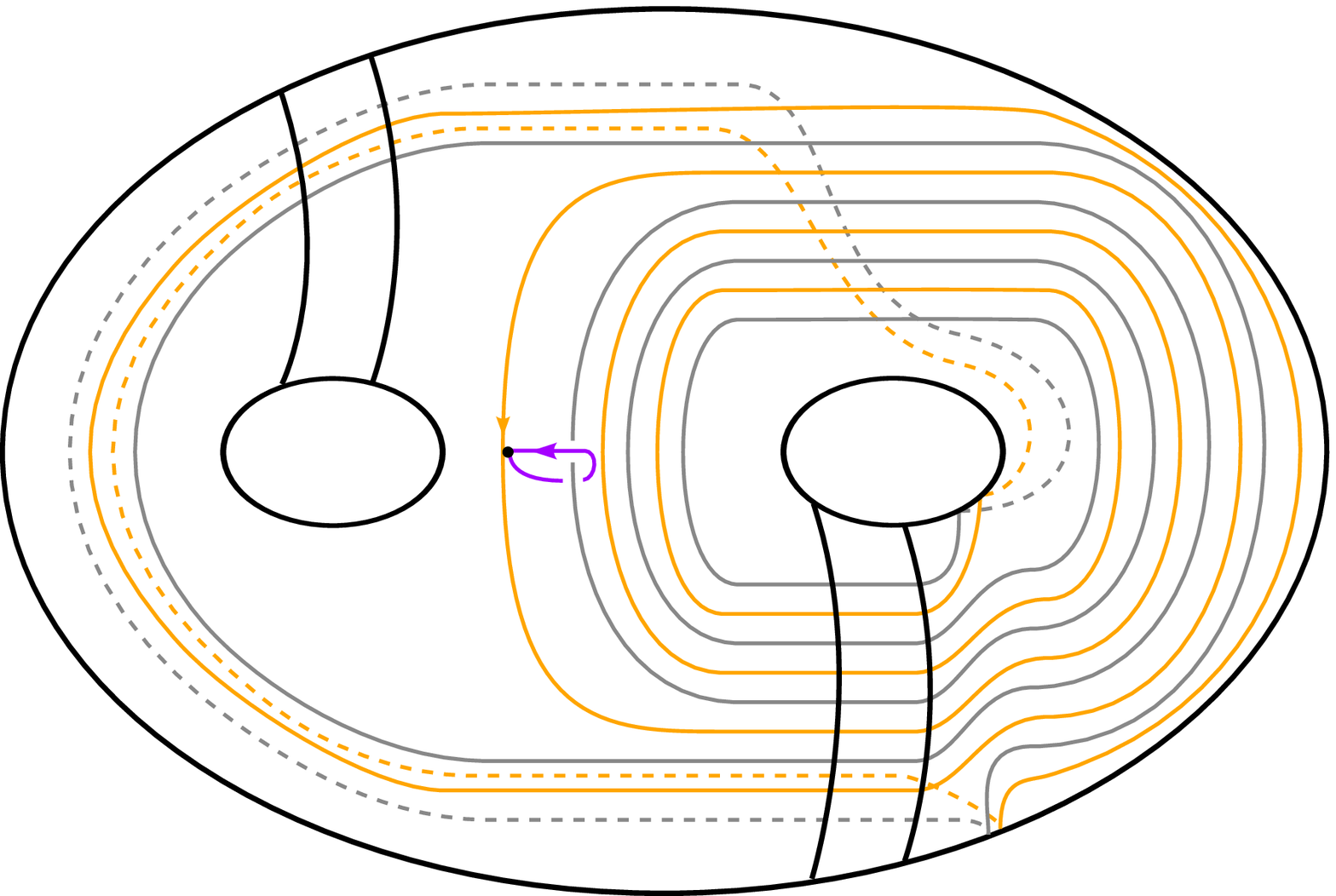}
\caption{Generators for the peripheral subgroup of $T_{p,pk-1}^{\ell,m}$.
Here, we show $T_{5,5k-1}^{2,m}$ as an example. The knot is shown in gray.}
\label{fig:pk-1_peripheral}
\end{figure}

\begin{proposition} \label{prop:pk-1}
For the $(p,pk-1,\ell,m)$-twisted torus knot,
\begin{enumerate}[label=(\alph*)]
\item The knot group is
\[G_{p,pk-1}^{\ell,m} = \langle a,b \mid a^{p-\ell}(a(b^{1-k(p-\ell)}a^{p-\ell})^m)^{\ell-1}a = b^{k(p-\ell)-1}(b^k(b^{1-k(p-\ell)}a^{p-\ell})^m)^{\ell-1}b^k \rangle.\]
\item The peripheral subgroup is generated by the meridian 
\[\mu = a^{-1}b^k\]
and the surface framing 
\[s = \mu^{p(pk-1)+\ell^2m}\lambda = a^{p-\ell-1}(a(b^{1-k(p-\ell)}a^{p-\ell})^m)^\ell a.\]
\end{enumerate}
\end{proposition}

\begin{proof}
Let $S^3 = U \cup_\Sigma V$ be the genus-two Heegaard splitting
of $S^3$ specified by Figure \ref{fig:pk-1_surfacegen}. 
Then $\pi_1(U)$ is the free group on the generators $a$ 
and $c$, and $\pi_1(V)$ is the free group on the generators $b$ 
and $d$ (see Figure~\ref{fig:pk-1_generators}). Using the Seifert-Van Kampen
Theorem, we can then express $G_{p,pk-1}^{\ell,m}$ as a free product with
amalgamation of $\pi_1(U)$ and $\pi_1(V)$. To do this, we need the images of
the generators of $\pi_1(\Sigma \setminus \nu (T_{p,pk-1}^{\ell,m}))$ under
inclusion into $\pi_1(U)$ and $\pi_1(V)$.

Now, $\Sigma \setminus \nu (T_{p,pk-1}^{\ell,m})$ is homotopy equivalent
to a twice-punctured genus-1 surface whose fundamental group is 
generated by the green, red, and blue loops in Figure 
\ref{fig:pk-1_surfacegen}. The green loop has image $c$ in 
$\pi_1(U)$ and image $d^m$ in $\pi_1(V)$, so we get the relation 
\[c = d^m.\]
Likewise, the red loop gives 
\[a^{p - \ell} = b^{(p - \ell)k - 1}d\]
and the blue loop gives
\[a^{p - \ell}(ac)^{\ell - 1}a = b^{(p - \ell)k - 1}(b^kd^m)^{\ell - 1}b^k.\]
Using the first two relations to solve for $c$ and $d$, we are
left with only one relation: 
\[a^{p-\ell}(a(b^{1-k(p-\ell)}a^{p-\ell})^m)^{\ell-1}a = b^{k(p-\ell)-1}(b^k(b^{1-k(p-\ell)}a^{p-\ell})^m)^{\ell-1}b^k.\]
Thus, we have
\[G_{p,pk-1}^{\ell,m} = \langle a,b \mid a^{p-\ell}(a(b^{1-k(p-\ell)}a^{p-\ell})^m)^{\ell-1}a = b^{k(p-\ell)-1}(b^k(b^{1-k(p-\ell)}a^{p-\ell})^m)^{\ell-1}b^k \rangle.\]

\begin{figure}
\labellist
\pinlabel $m$ at 331 92
\pinlabel twists at 331 87
\pinlabel \textcolor{a-color}{a} at 190 5
\pinlabel \textcolor{b-color}{b} at 150 40
\endlabellist
\centering
\includegraphics[width=\textwidth, height=6cm]{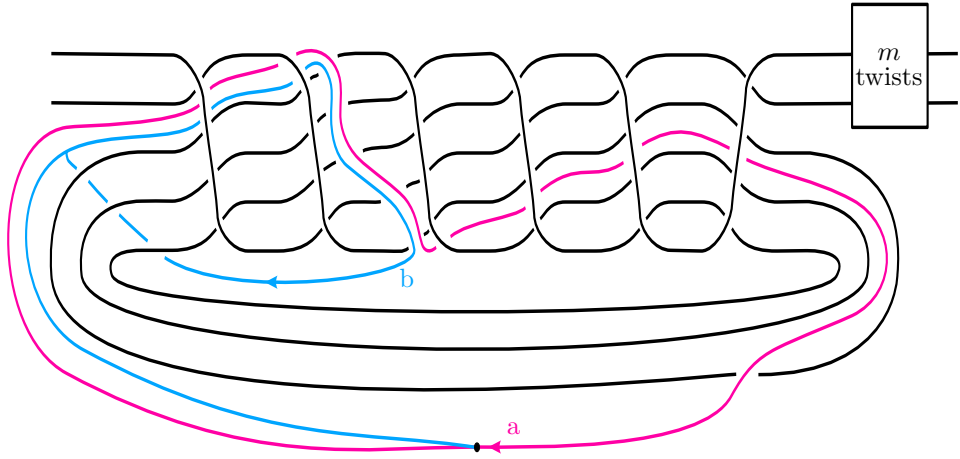}
\caption{Generators $a$ and $b$ for the knot group of $T_{5,5k-1}^{2,m}$.
We show $k = 1$ as a base case. Note that this braid is a view from the back
of the above handlebody diagrams.}
\label{fig:pk-1_ab}
\end{figure}
\begin{figure}
\labellist
\pinlabel $m$ at 327 92
\pinlabel twists at 327 87
\pinlabel \textcolor{ab-color}{$a^{-1}b$} at 190 7
\pinlabel \textcolor{mu-color}{$\mu$} at 160 8
\endlabellist
\centering
\includegraphics[width=\textwidth, height=6cm]{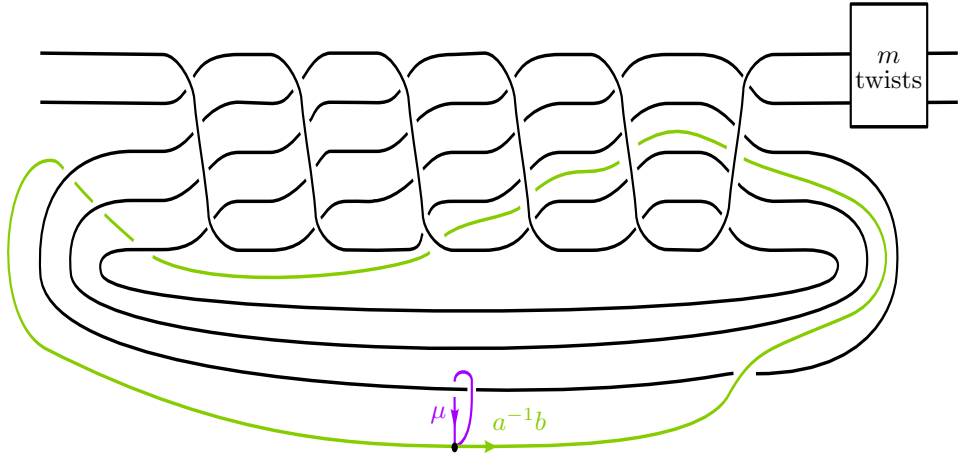}
\caption{Homotopy between $\mu$ and $a^{-1}b$ for the base case $k = 1$.}
\label{fig:pk-1_muhom}
\end{figure}

For the peripheral subgroup, we will compute the meridian $\mu$ and
the surface framing $s$ as specified in Figure \ref{fig:pk-1_peripheral}.
From Figure \ref{fig:pk-1_peripheral}, it is immediately clear that
\[s = a^{p-\ell}c(ac)^{\ell-1} a = a^{p-\ell-1}(a(b^{1-k(p-\ell)}a^{p-\ell})^m)^\ell a.\]
In order to compute $\mu$, we focus on the right half of the
handlebody in Figure \ref{fig:pk-1_peripheral}. This part of the knot
is shown along with $a$ and $b$ in Figure \ref{fig:pk-1_ab}.

As a base case, we consider $k = 1$. Figure \ref{fig:pk-1_muhom} 
demonstrates that the word $a^{-1}b$ is homotopic to $\mu$ in this case.
For larger $k$, we note that $b$ is homotopic to the core of one
full twist, as seen in Figure \ref{fig:pk-1_ab}. So, for each of
the $k - 1$ twists added to the $k = 1$ base case, we must append
one extra copy of $b$ to the end of the word $a^{-1}b$ in order to
create a word which is homotopic to $\mu$. Thus,
\[\mu = a^{-1}b^k.\]
Finally, we note that the linking number between $T_{p,q}^{\ell,m}$
and a push-off along $\Sigma$, by the construction of the twisted torus 
knot, is $pq + \ell^2 m$, which gives us
\[s = \mu^{p(pk-1) + \ell^2 m}\lambda.\]
\end{proof}

\begin{figure}
\labellist
\pinlabel ${\LARGE \underbrace{ }}$ at 125 187
\pinlabel {\small $m$ twists} at 125 173
\pinlabel $\LARGE \overbrace{ }$ at 333 160
\pinlabel {\small $k$ twists} at 333 175
\endlabellist
\centering
\includegraphics[width=\textwidth]{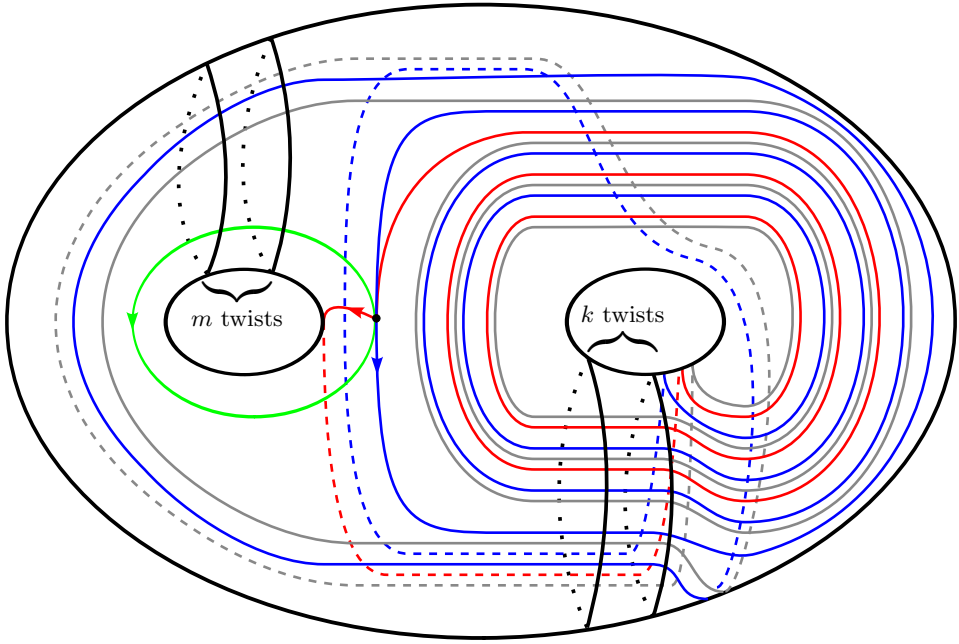}
\caption{Generators for the fundamental group of
$\Sigma \setminus \nu (T_{p,pk+1}^{\ell,m})$. Here, we show
$T_{5,5k+1}^{2,m}$ as an example. The knot is shown in gray.}
\label{fig:pk+1_surfacegen}
\end{figure}
\begin{figure}
\labellist
\pinlabel ${\LARGE \underbrace{ }}$ at 125 187
\pinlabel {\small $m$ twists} at 125 173
\pinlabel $\LARGE \overbrace{ }$ at 333 160
\pinlabel {\small $k$ twists} at 333 175
\pinlabel \textcolor{mu-color}{$\mu$} at 220 270
\pinlabel \textcolor{s-color}{s} at 180 260
\endlabellist
\centering
\includegraphics[width=\textwidth]{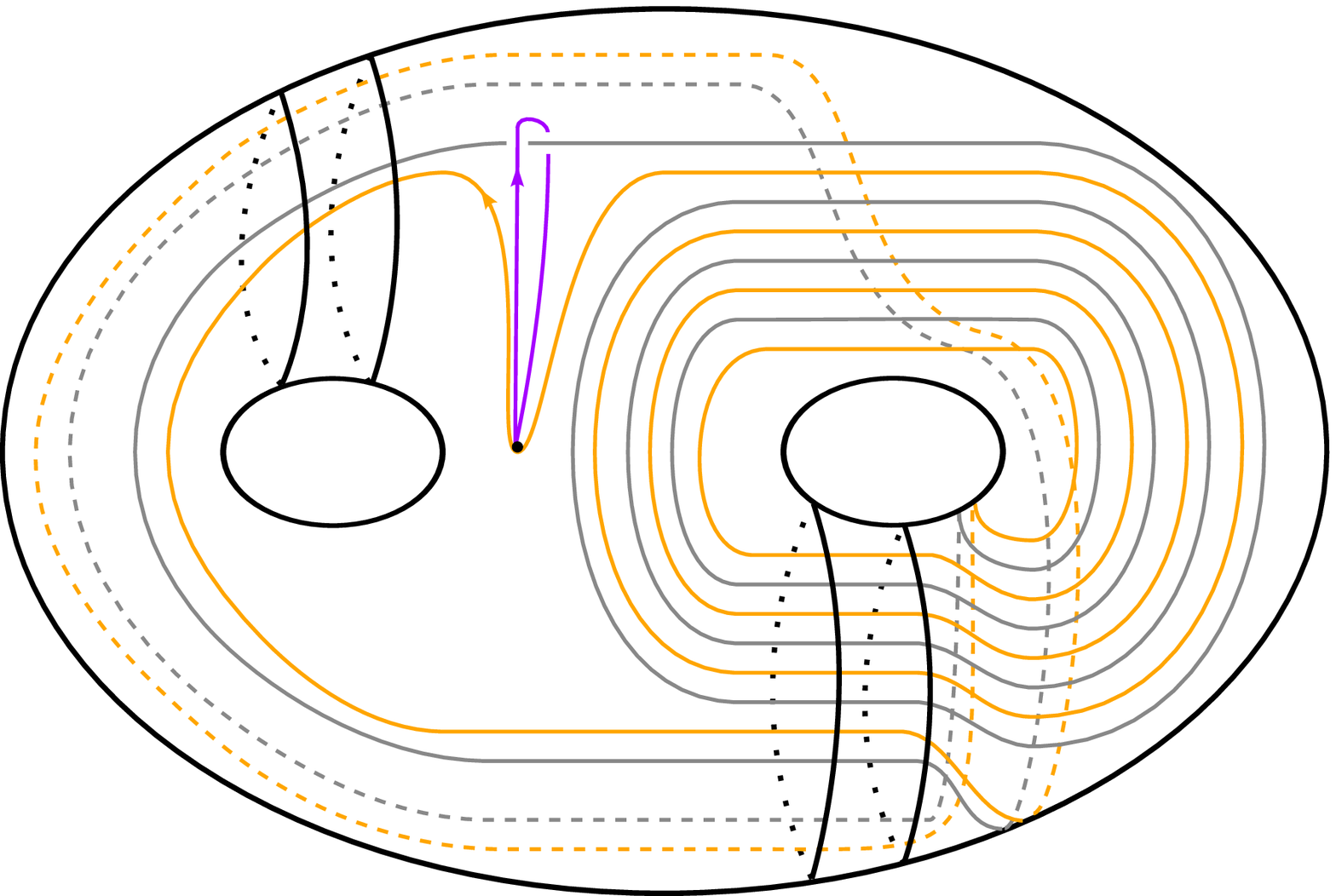}
\caption{Generators for the peripheral subgroup of $T_{p,pk+1}^{\ell,m}$.
Here, we show $T_{5,5k+1}^{2,m}$ as an example. The knot is shown in gray.}
\label{fig:pk+1_peripheral}
\end{figure}

\begin{proposition} \label{prop:pk+1}
For the $(p,pk+1,\ell,m)$-twisted torus knot,
\begin{enumerate}[label=(\alph*)]
\item The knot group is
\[G_{p,pk+1}^{\ell,m} = \langle a,b \mid a((b^{k(p-\ell)+1}a^{\ell-p})^ma)^{\ell-1}a^{p-\ell} = b^k((b^{k(p-\ell)+1}a^{\ell-p})^mb^k)^{\ell-1}b^{k(p-\ell)+1} \rangle.\]
\item The peripheral subgroup is generated by the meridian
\[\mu = b^{-k}a\]
and the surface framing 
\[s = \mu^{p(pk+1) + \ell^2m}\lambda = ((b^{k(p-\ell)+1}a^{\ell-p})^ma)^\ell a^{p-\ell}.\]
\end{enumerate}
\end{proposition}
\begin{proof}
Let $S^3 = U \cup_\Sigma V$ be the genus-two Heegaard splitting
of $S^3$ specified by Figure \ref{fig:pk+1_surfacegen}. We
use the same reasoning as in Proposition \ref{prop:pk-1} to write
$G_{p,pk+1}^{\ell,m}$ as a free product with
amalgamation of $\langle a,c \rangle$ and $\langle b,d \rangle$
(see Figure~\ref{fig:pk-1_generators}).

The generators for the fundamental group of 
$\Sigma \setminus \nu (T_{5,5k+1}^{2,m})$ are shown in Figure
\ref{fig:pk+1_surfacegen}. From the green loop, we get
\[c = d^m,\]
from the red loop, we get
\[a^{p - \ell} = d^{-1}b^{k(p - \ell)+1},\]
and from the blue loop, we get
\[(ac)^{\ell - 1}a^{p - \ell + 1} = (b^kd^m)^{\ell - 1}b^{k(p - \ell + 1) + 1}.\]
Using the first two relations to solve for $c$ and $d$, we are
left with only one relation: 
\[(a(b^{k(p-\ell)+1}a^{\ell-p})^m)^{\ell - 1}a^{p - \ell + 1} = (b^k(b^{k(p-\ell)+1}a^{\ell-p})^m)^{\ell - 1}b^{k(p - \ell + 1) + 1}.\]
Rewriting this slightly to create the same form as the group 
relation in Proposition \ref{prop:pk-1}, we get
\[G_{p,pk+1}^{\ell, m} = \langle a, b \mid a((b^{k(p-\ell)+1}a^{\ell-p})^ma)^{\ell - 1}a^{p - \ell} = b^k((b^{k(p-\ell)+1}a^{\ell-p})^mb^k)^{\ell - 1}b^{k(p - \ell) + 1} \rangle.\] 

\begin{figure}
\labellist
\pinlabel $m$ at 331 92
\pinlabel twists at 331 87
\pinlabel \textcolor{a-color}{a} at 190 5
\pinlabel \textcolor{b-color}{b} at 150 40
\endlabellist
\centering
\includegraphics[width=\textwidth, height=6cm]{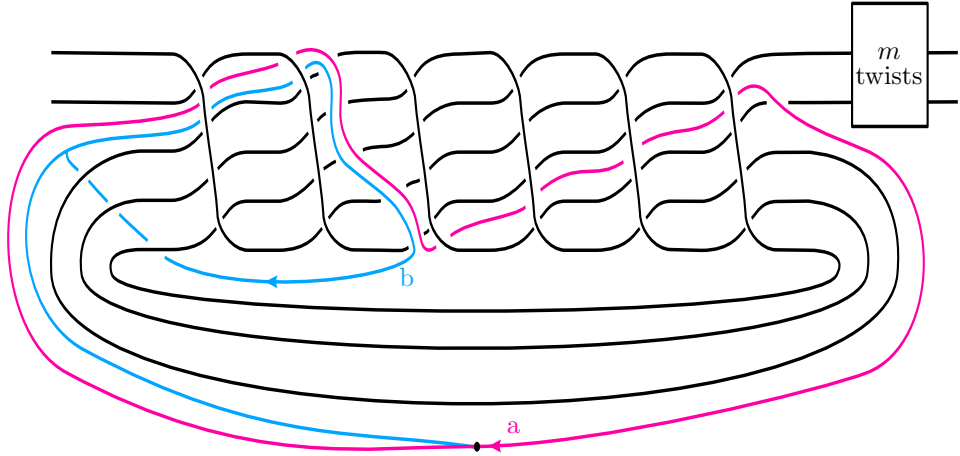}
\caption{Generators $a$ and $b$ for the knot group of $T_{5,5k+1}^{2,m}$.
Here, we show $k = 1$ as a base case. Note that this braid is a view from the
back of the above handlebody diagrams.}
\label{fig:pk+1_ab}
\end{figure}
\begin{figure}
\labellist
\pinlabel $m$ at 327 92
\pinlabel twists at 327 87
\pinlabel \textcolor{ab-color}{$b^{-1}a$} at 155 6
\pinlabel \textcolor{mu-color}{$\mu$} at 330 10 
\endlabellist
\centering
\includegraphics[width=\textwidth, height=6cm]{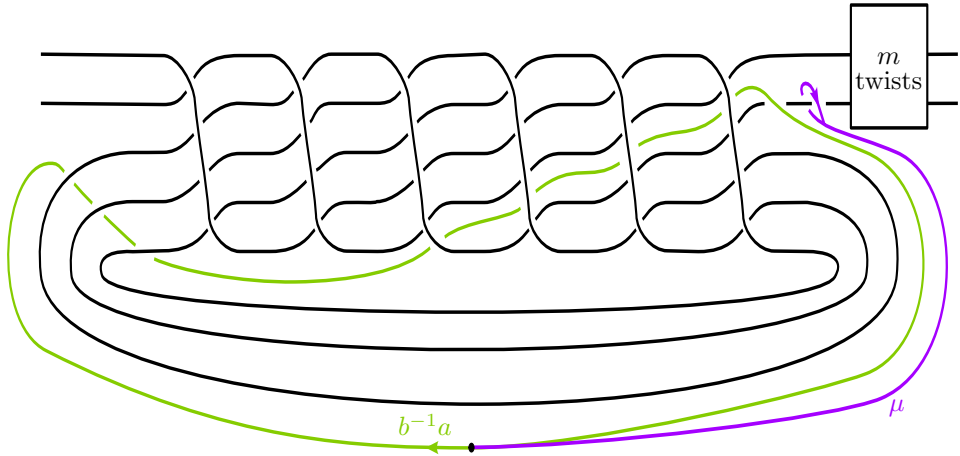}
\caption{Homotopy between $\mu$ and $b^{-1}a$ for the base case $k = 1$.}
\label{fig:pk+1_muhom}
\end{figure}

The reasoning for the peripheral subgroup also follows that of
Proposition \ref{prop:pk-1} very closely. It is immediate from Figure
\ref{fig:pk+1_peripheral} that
\[s = (ca)^{\ell}a^{p-\ell} = ((b^{k(p-\ell)+1}a^{\ell-p})^ma)^{\ell}a^{p-\ell}.\]

To compute $\mu$, we focus on the right half of the
handlebodies in Figure \ref{fig:pk+1_peripheral}. This part of the knot is
shown along with $a$ and $b$ in Figure \ref{fig:pk+1_ab}. We consider $k = 1$
as a base case: here, Figure \ref{fig:pk+1_muhom} demonstrates that the word
$b^{-1}a$ is homotopic to $\mu$. For larger $k$, we note that $b$ is homotopic
to the core of one full twist, as depicted in Figure \ref{fig:pk+1_ab}. So, for
each of the $k - 1$ twists added to the $k = 1$ base case, we must append one
extra copy of $b^{-1}$ to the start of the word $b^{-1}a$ in order to create a
word which is homotopic to $\mu$. Thus,
\[\mu = b^{-k}a.\]
Finally, the linking number between $T_{p,pk+1}^{\ell,m}$ and a
push-off along $\Sigma$ is $p(pk+~1)+~\ell^2m$, so that
\[s = \mu^{p(pk+1)+\ell^2m}\lambda.\]
\end{proof}

\section{Non-Left-Orderability of L-Space Twisted Torus Knots}
\label{sec:LO}
In this section, we prove Theorem~\ref{thm:CGHV1}.
The proof relies on applications of Theorem~\ref{thm:CGHVCond} to the twisted
torus knots specified in cases 1 and 2 of Theorem~\ref{thm:Vaf}.

First, we note that for all the $T_{p,q}^{\ell,m}$ which are L-space 
knots, we have knot groups with 2 generators, and we have computed expressions
for the meridian $\mu$ and the $pq+\ell^2m$-framed longitude $s$ (see Propositions
\ref{prop:pk-1} and \ref{prop:pk+1}). To apply Theorem \ref{thm:CGHVCond},
then, we have to find generators $x$ and $y$ such that
\begin{enumerate}
\item \label{enumit:step1} 
for any homomorphism $\Phi: G_{p,q}^{\ell,m}(r) \rightarrow \Homeo+(\R)$, $xt > t$
for all $t \in \R$ implies $yt \geq t$ for all $t \in \R$; and
\item \label{enumit:step2}
$x$ is (a conjugate of) $\mu$, and (the corresponding conjugate of) $s$ can
be written with only positive powers of $x$ and $y$ and at least one $x$. 
\end{enumerate}
Then, we can apply Theorem~\ref{thm:CGHVCond} to conclude that $r$-surgery on 
$T_{p,q}^{\ell,m}$ yields a 3-manifold with non-left-orderable
fundamental group for all $r \geq pq + \ell^2 m$.

We begin with two lemmas that verify condition \ref{enumit:step1} of
the above list for the first two cases in Theorem~\ref{thm:Vaf}.
\begin{lemma}
\label{lem:pk-1_step1}
Consider $G_{p,pk-1}^{\ell,m}$ as in Proposition~\ref{prop:pk-1}. Let
$x = a^{-1}b^k$ and $y = b^{1-k}a$. Then, $x$ is a meridian of
$T_{p,pk-1}^{\ell,m}$, $x$ and $y$ generate $G_{p, pk-1}^{\ell,m}$, and 
for any homomorphism $\Phi: G_{p, pk-1}^{\ell, m}(r) \rightarrow \Homeo+(\R)$, $xt > t$ for all real
$t$ implies that $yt \geq t$ for all real $t$.
\end{lemma} 
\begin{proof}
By Proposition~\ref{prop:pk-1}, $x$ is a meridian of $T_{p,pk-1}^{\ell,m}$. Moreover,
$b = yx$ and $a = (yx)^{k-1}y$, so $x$ and $y$ generate $G_{p, pk-1}^{\ell,m}$.
Next, we examine the group relation from Proposition~\ref{prop:pk-1} in terms
of $x$ and $y$:
\[((yx)^{k-1}y)^{p-\ell}((yx)^{k-1}yC^m)^{\ell - 1}(yx)^{k-1}y = (yx)^{k(p-\ell)-1}((yx)^kC^m)^{\ell - 1}(yx)^k\]
where $C = b^{1-k(p-\ell)}a^{p-\ell}$.
If we assume $xt > t$ for all $t$, we can add $x$ anywhere we want
on one side of the relation to get a strict inequality on any $t$. In
symbols, if we have $w_1t=w_2t$ and $w_3w_4 = w_1$, then we know
that $xw_4t > w_4t$, so $w_3xw_4t > w_3w_4t = w_1t = w_2t$. Adding
$x$ multiple times to the left side of the relation, we get
\[(yx)^{k(p-\ell)-1}y((yx)^kC^m)^{\ell - 1}(yx)^kt > (yx)^{k(p-\ell)-1}((yx)^kC^m)^{\ell - 1}(yx)^kt.\]
Since every word corresponds to a homeomorphism on $\R$, we know that
for all $t' \in \R$ there exists $t$ such that
$((yx)^kC^m)^{\ell - 1}(yx)^kt = t'$. Thus, we have
\[(yx)^{k(p-\ell)-1}yt' > (yx)^{k(p-\ell)-1}t'\]
which implies that 
\[yt' > t'\]
for all $t' \in \R$.
\end{proof}
\begin{lemma}
\label{lem:pk+1_step1}
Consider $G_{p,pk+1}^{\ell,m}$ as in Proposition~\ref{prop:pk+1}. Let 
$x = b^{-k}a$ and $y = a^{-1}b^{k+1}$. Then, $x$ is a meridian
of $T_{p,pk+1}^{\ell,m}$, $x$ and $y$ generate $G_{p, pk+1}^{\ell,m}$, and
for any homomorphism $\Phi: G_{p, pk+1}^{\ell, m}(r) \rightarrow \Homeo+(\R)$, $xt > t$ for all real $t$
implies $yt \geq t$ for all real $t$.
\end{lemma}
\begin{proof}
By Proposition~\ref{prop:pk+1}, $x$ is a meridian of $T_{p,pk+1}^{\ell,m}$. Moreover,
$b = xy$ and $a = (xy)^kx$, so $x$ and $y$ generate $G_{p, pk+1}^{\ell,m}$. Using the group relation
from Proposition \ref{prop:pk+1} and rewriting in terms of $x$ and $y$, we get
\[(xy)^kx(C^m(xy)^kx)^{\ell-1}((xy)^kx)^{p-\ell} = (xy)^k(C^m(xy)^k)^{\ell-1}(xy)^{k(p-\ell)+1}\]
where $C = b^{k(p-\ell)+1}a^{\ell-p}$. Assume $xt > t$ for all 
$h \in \R$. Then, adding $x$'s to the right-hand side of the above
equation gives us
\[(xy)^kx(C^m(xy)^kx)^{\ell-1}((xy)^kx)^{p-\ell}t < (xy)^kx(C^m(xy)^kx)^{\ell-1}((xy)^kx)^{p-\ell}yt.\]
Now, the word $(xy)^kx(C^mx(xy)^kx)^{\ell-1}((xy)^kx)^{p-\ell}$
corresponds to an order-preserving homeomorphism on $\R$. By
order preservation, the above inequality can only be true if
\[t < yt\]
for all $t \in \R$.
\end{proof}
Theorem~\ref{thm:CGHV1} follows from the next four propositions,
which verify condition \ref{enumit:step2} of the list outlined above.
\begin{proposition}
$M_{p,pk-1}^{p-1,m}(r)$ has a non-left-orderable fundamental
group for $r \geq p(pk-1) +(p-1)^2m$.
\end{proposition}
\begin{proof}
To satisfy the hypotheses of Theorem~\ref{thm:CGHVCond}, by Lemma~\ref{lem:pk-1_step1},
it is sufficient to check that $s$ can be written with only positive powers
of $x$ and $y$, with $x$ and $y$ as defined in the lemma.
In terms of $x$ and $y$, the expression for $s$ in Proposition~\ref{prop:pk-1}
with $\ell=p-1$ becomes
\[s = (a(b^{1-k}a)^m)^{p-1}a = ((yx)^{k-1}y^{m+1})^{p-1}(yx)^{k-1}y\]
which contains only positive powers of $x$ and $y$ and at least one
$x$. We can then apply Theorem \ref{thm:CGHVCond}, noting that 
$\lambda = \mu^{-p(pk-1) - (p-1)^2m}s$.
\end{proof}
\begin{proposition}
$M_{p,pk-1}^{p-2,1}(r)$ has a non-left-orderable fundamental
group for $r \geq p(pk-1) + (p-2)^2$
\end{proposition}
\begin{proof}
It is again sufficient to check that $s$ only contains positive powers of
$x$ and $y$, with $x$ and $y$ as in Lemma~\ref{lem:pk-1_step1}. 
Now, we consider the expression for $s$ from Proposition~\ref{prop:pk-1}
with $\ell=p-2$ and $m=1$ and rewrite it using the group relation:
\[s = a^{-1}a^2(ab^{-2k+1}a^2)^{p-2}a = a^{-1}b^{2k-1}(b^{1-k}a^2)^{p-2}a.\]
In terms of $x$ and $y$, this becomes
\[s = x(yx)^{k-1}(y(yx)^{k-1}y)^{p-2}(yx)^{k-1}y\]
which contains only positive powers of $x$ and $y$ and at least
one $x$. We can then apply Theorem \ref{thm:CGHVCond}, noting that
$\lambda = \mu^{-p(pk-1) - (p-2)^2}s$.
\end{proof}
\begin{proposition}
$M_{p,pk+1}^{p-1,m}(r)$ has a non-left-orderable fundamental
group for $r \geq p(pk+1) + (p-1)^2m$.
\end{proposition}
\begin{proof}
To satisfy the hypotheses of Theorem~\ref{thm:CGHVCond}, by
Lemma~\ref{lem:pk+1_step1}, it is sufficient to check that
$s$ can be written with only positive powers of $x$ and $y$, with $x$ and $y$
as defined in the lemma.
Using the expression for $s$ in Proposition~\ref{prop:pk+1} with $\ell=p-1$, we get
\[s = ((b^{k+1}a^{-1})^ma)^{p-1}a = (b^{k+1}(a^{-1}b^{k+1})^{m-1})^{p-1}a = ((xy)^{k+1}y^{m-1})^{p-1}(xy)^kx\]
which contains only positive powers of $x$ and $y$, with at
least one $x$. We can then apply Theorem \ref{thm:CGHVCond}, noting that
$\lambda = \mu^{-p(pk+1)-(p-1)^2m}s$.
\end{proof}
\begin{proposition}
$M_{p,pk+1}^{p-2,1}(r)$ has a non-left-orderable fundamental
group for $r \geq p(pk+1) + (p-2)^2$.
\end{proposition}
\begin{proof}
It is again sufficient to check that $s$ can be written with only
positive powers of $x$ and $y$, with $x$ and $y$ as in Lemma~\ref{lem:pk+1_step1}.
Using the expression for $s$ in Proposition~\ref{prop:pk+1} with $\ell=p-2$
and $m=1$, we get
\[s = (b^{2k+1}a^{-1})^{p-2}a^2 = b^{2k+1}(a^{-1}b^{2k+1})^{p-3}a = (xy)^{2k+1}(y(xy)^k)^{p-3}(xy)^kx\]
which contains only positive powers of $x$ and $y$ and at
least one $x$. We can then apply Theorem \ref{thm:CGHVCond}, noting that
$\lambda = \mu^{-p(pk+1)-(p-2)^2}s$.
\end{proof}

\bibliography{REU}{}
\bibliographystyle{amsalpha}
\end{document}